\documentclass[12pt,a4paper]{article}
\usepackage{pslatex}
\usepackage{fancyhdr}
\usepackage{graphicx}
\usepackage{geometry}
\usepackage{amssymb}
\usepackage{amsmath}
\usepackage{graphics}
\usepackage{epsfig}
\usepackage{mathrsfs}
\usepackage{color}
\usepackage{showlabels}

\newtheorem{claim}{\bf \t}[part]

\newtheorem{Definition}{Definition}[part]

\newtheorem{Lemma}{Lemma}[part]

\numberwithin{Assumption}{section}
\numberwithin{Corollary}{section}
\numberwithin{Definition}{section}
\numberwithin{equation}{section} \numberwithin{Example}{section}
\numberwithin{Lemma}{section} \numberwithin{Proposition}{section}
\numberwithin{Remark}{section} \numberwithin{Theorem}{section}

\def \E{{\mathbb{E}}}
\def \D{{\mathbb{D}}}
\def \P{{\mathbb{P}}}

\def \L{{\mathbb{L}}}

\def\vare{\varepsilon}

\def\text#1{{\rm #1}}
\def \eref#1{\hbox{(\ref{#1})}}

\definecolor{dg}{rgb}{0, 0.5, 0}

\def \eref#1{\hbox{(\ref{#1})}}

\newcommand{\beq}{\begin{equation}}
\newcommand{\eeq}{\end{equation}}
\newcommand{\bea}{\begin{eqnarray}}
\newcommand{\eea}{\end{eqnarray}}
\newcommand{\beas}{\begin{eqnarray*}}
\newcommand{\eeas}{\end{eqnarray*}}

\newcommand{\EE}{{\mathbb E}}

\newcommand{\R}{{\mathbb R}}

\newenvironment{proof}[1][Proof]{\textbf{#1.} }{\ \rule{0.5em}{0.5em}}

\newtheorem{theorem}{Theorem}[section]

\newtheorem{definition}[theorem]{Definition}

\newtheorem{proposition}[theorem]{Proposition}

\newtheorem{remark}[theorem]{Remark}

\let\Section=\section
\def\section{\setcounter{equation}{0}\Section}

\begin{document}
\title{ On optimal mean-field type control problems of stochastic systems with jump processes under partial information\thanks{Y. Hu is partially supported by a grant from the Simons Foundation No.209206.
D. Nualart is supported by the NSF grant DMS1208625.
Q. Zhou is supported by the National Natural
Science Foundation of China (No 11001029 and 11371362) and the
Fundamental Research Funds for the Central Universities (No BUPT2012RC0709). }}

\author{ \text{Yaozhong Hu}$^1$ \footnote{E-mail: hu@math.ku.edu}\;\  \ \ \text{David Nualart}$^2$ \footnote{E-mail: nualart@math.ku.edu}\  \ \ \text{Qing Zhou}$^3$ \footnote{Corresponding author. E-mail: zqleii@bupt.edu.cn}
\\
 \small  1, 2. Department of Mathematics, University of Kansas, Lawrence, Kansas, 66045  USA   \\
 \small 3. School of Science, Beijing University of Posts and Telecommunications, Beijing 100876, China}\,

\date{}
\maketitle

\begin{abstract}
This paper considers the problem of partially observed optimal control for forward
 stochastic systems which are driven by Brownian motions and an independent Poisson
 random measure with a feature that the cost functional is of mean-field type.
When all the system coefficients and the objective performance functionals are
allowed to be random, possibly non-Markovian,  Malliavin calculus is employed to
derive a maximum principle for the optimal control of such a system where the adjoint
 process is explicitly expressed. We also investigate the mean-field type optimal control
  problems for systems driven by mean-field type stochastic differential equations
  (SDEs in short) with jump processes, in which the coefficients contain not only
  the state process but also its marginal distribution under partially observed
  information. The maximum principle is established using convex variational technique
  with an illustrating example about linear-quadratic optimal control.
\end{abstract}

\vspace{.08in} \noindent \textbf{Keywords:}  maximum principle; mean-field type; partial information; Girsanov's theorem; forward stochastic differential equations; Malliavin calculus; jump diffusion

\vspace{.08in} \noindent \textbf{\bf AMS subject classification} 60H10; 60HXX; 60H07; 60J75

\section{Introduction }
Let $T>0$ be a fixed time horizon and let $(\Omega, \mathscr{F},
\P_0)$ be a probability space equipped with a  right continuous
filtration  $(\mathscr{F}_t)_{t\in [0,T]}$ satisfying the usual
conditions (we use the notation $\P_0$ here to reserve $\P$ for a
future use). Let
 $((W_1(t),W_2(t))\,, 0\le t\le T)$  be a two dimensional $\mathscr{F}_t$-Brownian motion and let $N(dt,dz)$ be
 a $\mathscr{F}_t$-Poisson  random measure on $[0, T]\times \R_0$ with  intensity  $\mu(dz)$, independent of the Brownian
 motion $W_1$ and $W_2$,  where $\R_0=\R\setminus\{0\}$.   We denote
the compensated  Poisson measure  by  $\tilde{N}(dt,dz):=N(dt,dz)-\mu(dz)dt$.




In this paper we shall study the mean field  optimal control problems
which have the following characteristics.  The state equation is
given by the following mean field stochastic differential equation
with jumps:
\begin{eqnarray}\label{Anew}\left\{
\begin{array}{lll}
    dx^\nu(t)\!\!\!\!\!\!&= b(t,x^\nu(t),\E_0[x^\nu(t)],v(t))dt+\sigma(t,x^\nu(t),\E_0[x^\nu(t)],v(t))dW_1(t)\\
& \quad+\displaystyle\int_{ \R_0}\gamma(t,x^\nu(t-),\E_0[x^\nu(t)],v(t),z)\tilde{N}(dt,dz),\\
x^\nu(0)\!\!\!\!\!\!&=x_0,  \quad t\in[0,T], \\
\end{array}%
\right.
\end{eqnarray}
where $v(\cdot)$ is a control process taking values in a nonempty,
closed convex subset $U\subseteq\R$ and the expectation $\E_0 $ is related to the probability measure $ \P_0 $.  The conditions on the
coefficients   will be made specific in Section 4.  To describe the
conditions on the control $\nu$, we assume the state process
$x^\nu(t)$ is not completely observable. Instead, it is partially
observed and the observation is corrupted with noise. The observation equation is given by the following
equation:
\begin{eqnarray}&&\left\{%
\begin{array}{lll}  dY(t) &=&  h(t,x^v(t))dt+dW_2(t) ,\\
Y(0)&=&0,
\end{array}%
\right.\label{observation}
\end{eqnarray}
Thus the control process $v(t)$ will be an ${\mathscr{F}} _t^Y$-adapted
processes. More precisely, we give the following definition of
admissible controls.
\begin{definition}\label{d.adm} Let $U \subseteq\R$ be  a nonempty, closed
and convex subset which  will be the range of control
$v$.  Let ${\mathscr{F}}_t^Y=\sigma(Y_s, 0\le s\le t)$ be the
$\sigma$-algebra generated by $Y$. A control process $v:
[0,T]\times\Omega\rightarrow U$ is called admissible if $v(t)$ is
${\mathscr{F}}_t^Y$-adapted  and $\sup\limits_{0\leq t\leq
T}\E_0|v(t)|^2<\infty.$ The set of all admissible controls is denoted
by ${\mathscr U}_{ad}$.
\end{definition}

We introduce the following  cost functional
\begin{eqnarray}\label{performance}J(v(\cdot))&=& \E_0 \left[\int_0^T l(t, x^v(t), \E_0 [f(x^v(t))], v(t)  )dt
+\phi(x^v(T),\E_0 [g(x^v(T))] )\right]\,,
\end{eqnarray}
where $l:[0,T]\times\R\times\R\times U\rightarrow\R$ and
$\phi:\R\times\R\rightarrow\R$ are given mappings. $\E_0$ is the expectation with
respect to the probability measure $\P_0$. $f:\R\to\R$ and $g:\R\to\R$ are given functions such that
$\E_0[|f(x^v(t))|]<\infty,$ for all $t$ and $\E_0[|g(x^v(T))|]<\infty.$

Now  we can state  our  mean-field type control (MFC) problem  as
follows.

\medskip
\noindent{\it Problem (MFC)}: Find $u(\cdot)\in {\mathscr{U}}_{ad}$
(if it exists) such that
\begin{eqnarray*}{\label{Jmin}} J(u(\cdot))=\min_{v(\cdot)\in
{\mathscr{U}}_{ad}}J(v(\cdot))\,.
\end{eqnarray*}
Our objective in this paper is to establish a maximum principle for
the optimal control to satisfy. This will be given in Section 4.
However, we shall pay a particular attention to the case when  the
state equation does not contain the mean field, namely, when the
state equation is given by the following equation
\begin{eqnarray}
&&\left\{
\begin{array}{lll}
    dx^v(t)\!\!\!\!\!\!&= b(t,x^v(t),v(t))dt+\sigma(t,x^v(t),v(t))dW_1(t)\\
& \quad+\displaystyle\int_{\R_0}\gamma(t,x^v(t-),v(t),z)\tilde{N}(dt,dz),\\
x^v(0)\!\!\!\!\!\!&=x_0,  \quad t\in[0,T]\,.   \\
\end{array}%
\right. \label{A}
\end{eqnarray}
In this case we can use Malliavin calculus to obtain more explicit
form of the maximum principle. This is done in Section 3.
It seems that this paper is the first
to study the problem of minimizing \eref{performance}
subject to  state constraint  \eref{Anew} and observation
constraint  \eref{observation}.   When full information is available,
the maximum principle is obtained in    \cite{HA},   \cite{Yang}  and references therein.

The work closely related to ours  is the work \cite{MBZ},  where the authors
have already used    Malliavin calculus to obtain the maximum principle.
The difference is that their partial information flow is   given  by a
fixed sub $\sigma$-algebra,  independent of the control. More precisely,
they assume $Y_t\in {\mathscr {E}}_t$, where
$ {\mathscr {E}}_t= {\mathscr {F}}_{(t-\delta)^+}$ for some fixed $\delta>0$.
In our model,  we assume the control process $u_t\in {\mathscr{F}}_t^Y,$ where $Y$
depends on $u$.  So ${\mathscr{F}}^Y$ also depends on the control $u$.

The rest of this paper is organized as follows. Section 2 gives a brief review of Malliavin calculus for L\'evy processes. In
Section 3, a maximum principle for mean-field type optimal control is derived using Malliavin calculus. We also give the applications in finance. In Section 4,
 we obtain the stochastic maximum principle for jump-diffusion mean-field SDEs by convex variation with an example about linear-quadratic optimal control.



\section{A brief review of Malliavin calculus for L\'evy processes}

In this section, we recall the basic definitions and properties of Malliavin calculus for Brownian motion $W(\cdot)$ and $N(ds,dz)$ related to this paper, for reader's convenience.

Let $L^2(\lambda^n)$ be the space of deterministic real functions $f$ such that
$$\|f\|_{L^2(\lambda^n)}=\left(\int_{[0,T]^n}f^2(t_1,t_2,\cdots,t_n)dt_1dt_2\cdots dt_n\right)^{1/2}<\infty,$$
where $\lambda(dt)$ denotes the Lebesgue measure on $[0,T].$

Let $L^2((\lambda\times\mu)^n)$ be the space of deterministic real functions $f$ such that
$$\|f\|_{L^2((\lambda\times\mu)^n)}=\left(\int_{([0,T]\times {\R_0})^n}f^2(t_1,z_1,t_2,z_2,\cdots,t_n,z_n)dt_1\mu(dz_1)dt_2\mu(dz_2)\cdots dt_n\mu(dz_n)\right)^{1/2}<\infty.$$
$L^2(\lambda\times \P)$ can be similarly denoted.

A general reference for this presentation is \cite{BNLOP}, \cite{NMOP} and \cite{N}. See also the book \cite{NOP}.

\subsection{Malliavin calculus for $W(\cdot)$}
A natural starting point is the Wiener-It\^o chaos expansion theorem, which states that any $F\in L^2(\mathscr{F}_T, \P)$
(where in this case $\mathscr{F}_t=\mathscr{F}_t^{W}$ is the $\sigma$-algebra generated by $W(s); 0\leq s\leq t$) can be written as
\begin{eqnarray}\label{F}F=\sum\limits_{n=0}^\infty I_n(f_n),
\end{eqnarray}
for a unique sequence of symmetric deterministic functions $f_n\in L^2(\lambda^n)$, where $\lambda$ is a Lebesgue measure on $[0,T]$ and
\begin{eqnarray*}I_n(f_n)=n!\int_0^T\int_0^{t_n}\cdots\int_0^{t_2}f_n(t_1,\cdots,t_n)dW(t_1)\cdots dW(t_n)
\end{eqnarray*}
(the $n$-times iterated integral of $f_n$ with respect to $W(\cdot)$) for $n=1,2,\cdots$ and $I_0(f_0)=f_0$ when $f_0$ is a constant.

Moreover, we have the isometry
\begin{eqnarray*}\E[F^2]=\|F\|^2_{L^2(\P)}=\sum\limits_{n=0}^{\infty}n!\|f_n\|^2_{L^2(\lambda^n)}.
\end{eqnarray*}
\begin{Definition} (Malliavin derivative $D_t$). Let ${\cal D}_{1,2}^{(W)}$ be the space of all $F\in L^2(\mathscr{F}_T,\P)$ such that its chaos expansion
(\ref{F}) satisfies
\begin{eqnarray*}\|F\|^2_{{\cal D}_{1,2}^{(W)}}:=\sum\limits_{n=1}^{\infty}nn!\|f_n\|^2_{L^2(\lambda^n)}<\infty.
\end{eqnarray*}
For $F\in {\cal D}_{1,2}^{(W)}$ and $t\in [0,T]$, we define the {\it Malliavin derivative} of $F$ at $t$ (with respect to $W(\cdot)$), $D_tF$, by
\begin{eqnarray*}D_tF=\sum\limits_{n=1}^\infty nI_{n-1}(f_n(\cdot,t)),
\end{eqnarray*}
where the notation $I_{n-1}(f_n(\cdot,t))$ means that we apply the $(n-1)$-times iterated integral to the first $n-1$ variables $t_1,\cdots, t_{n-1}$
of $f_n(t_1,t_2,\cdots,t_n)$ and keep the last variable $t_n=t$ as a parameter.\end{Definition}


\vspace{0.2cm}
Some basic properties of the Malliavin derivative $D_t$ are the following:

(i) Chain rule (\cite{N}, page 29)

Suppose $F_1,\cdots, F_m\in {\cal D}_{1,2}^{(W)}$ and that $\psi:\R^m\rightarrow \R$ is $C^1$ with bounded partial derivatives. Then

 $\psi(F_1,\cdots,F_m)\in {\cal D}_{1,2}^{(W)}$ and
\begin{eqnarray}\label{Chain rule}D_t\psi(F_1,\cdots,F_m)=\sum\limits_{i=1}^m\frac{\partial \psi}{\partial x_i}(F_1,\cdots,F_m)D_tF_i.
\end{eqnarray}

(ii) Integration by parts/duality formula (\cite{N}, page 35)

Suppose $h(t)$ is $\mathscr{F}_t$-adapted with $\E[\int_0^Th^2(t)dt]<\infty$ and let $F\in {\cal D}_{1,2}^{(W)}.$ Then
\begin{eqnarray}\label{part}\E\left[F\int_0^Th(t)dW(t)\right]=\E\left[\int_0^Th(t)D_t Fdt\right].
\end{eqnarray}

\subsection{Malliavin calculus for ${N}(\cdot)$}
The construction of a stochastic derivative/Malliavin derivative in the pure jump martingale case follows the same lines as in the Brownian motion case.
In this case, the corresponding Wiener-It\^o chaos expansion theorem states that any $F\in L^2(\mathscr{F}_T,\P)$ (where in this case $\mathscr{F}_t=\mathscr{F}_t^{{N}}$ is the $\sigma$-algebra generated by $\int_0^s\int_{A}{N}(dr,dz); 0\leq s\leq t$, $A\in \mathscr{B}(\R_0) $) can be written as
\begin{eqnarray}\label{FN} F=\sum\limits_{n=0}^{\infty}I_n(f_n); \ \ f_n\in \hat{L}^2((\lambda\times\mu)^n),
\end{eqnarray}
where $\mathscr{B}(\R_0)$ is the Borel $\sigma$-field generated by the open subset $O$ of $\R_0,$ whose closure $\bar{O}$ does not contain the point 0, and $\hat{L}^2((\lambda\times\mu)^n)$ is the space of functions $f_n(t_1,z_1,\cdots,t_n,z_n); \ t_i\in [0,T],$ $z_i\in \R_0$ such that
$f_n\in L^2((\lambda\times\mu)^n)$ and $f_n$ is symmetric with respect to the pairs of variables $(t_1,z_1),\cdots, (t_n, z_n).$

It is important to note that in this case the $n$-times iterated integral $I_n(f_n)$ is taken with respect to $\tilde{N}(dt,dz)$. Thus, we define
\begin{eqnarray*} I_n(f_n)=n!\int_0^T\int_{\R_0}\int_0^{t_n}\int_{\R_0}\cdots\int_0^{t_2}\int_{\R_0}f_n(t_1,z_1,\cdots,t_n,z_n)\tilde{N}(dt_1,dz_1)
\cdots\tilde{N}(dt_n,dz_n),
\end{eqnarray*}
for $\ f_n\in L^2((\lambda\times\mu)^n).$

Then It\^o isometry for stochastic integrals with respect to $\tilde{N}(dt,dz)$ gives the following isometry for the chaos expansion:
\begin{eqnarray*} \|F\|^2_{L^2(\P)}=\sum\limits_{n=0}^\infty n!\|f_n\|^2_{ L^2((\lambda\times\mu)^n)}.
\end{eqnarray*}
As in the Brownian motion case, we use the chaos expansion to define the Malliavin derivative. Note that in this case there are two parameters $t,z,$ where $t$
represents time and $z\neq 0$ represents a generic jump size.
\begin{Definition}(Malliavin derivative $D_{t,z}$) (\cite{BNLOP}, \cite{NMOP}) Let ${\cal D}_{1,2}^{(\tilde{N})}$ be the space of all $F\in L^2(\mathscr{F}_T,\P)$ such that its chaos expansion (\ref{FN}) satisfies
\begin{eqnarray*} \|F\|^2_{{\cal D}_{1,2}^{(\tilde{N})}}=\sum\limits_{n=1}^\infty n n!\|f_n\|^2_{ L^2((\lambda\times\mu)^n)}<\infty.
\end{eqnarray*}
For $F\in {\cal D}_{1,2}^{(\tilde{N})},$ we define the Malliavin derivative of $F$ at $(t,z)$ (with respect to $N(\cdot))$, $D_{t,z}F,$ by
\begin{eqnarray*} D_{t,z}F=\sum\limits_{n=1}^\infty nI_{n-1}(f_n(\cdot,t,z)),
\end{eqnarray*}
where $I_{n-1}(f_n(\cdot,t,z))$ means that we perform the $(n-1)$-times iterated integral with respect to $\tilde{N}$ to the first $n-1$ variable pairs $(t_1,z_1),\cdots, (t_n,z_n)$, keeping $(t_n,z_n)=(t,z)$ as a parameter.
\end{Definition}


The properties of $D_{t,z}$ corresponding to the properties (\ref{Chain rule}) and (\ref{part}) of $D_t$ are the following:

(i) Chain rule (\cite{NMOP}, \cite{I})

Suppose $F_1,\cdots, F_m\in {\cal D}_{1,2}^{(\tilde{N})}$ and that $\varphi:\R^m\rightarrow \R$ is continuous and  bounded. Then

 $\varphi(F_1,\cdots,F_m)\in {\cal D}_{1,2}^{(\tilde{N})}$ and
\begin{eqnarray}\label{Chain rule2}D_{t,z}\varphi(F_1,\cdots,F_m)=\varphi(F_1+D_{t,z}F_1,\cdots,F_m+D_{t,z}F_m)-\varphi(F_1,\cdots,F_m).
\end{eqnarray}

(ii) Integration by parts/duality formula (\cite{NMOP})

Suppose $\Psi(t,z)$ is $\mathscr{F}_t$-adapted and $\E[\int_0^T\int_{\R_0}\Psi^2(t,z)\mu(dz)dt]<\infty$ and let $F\in {\cal D}_{1,2}^{(\tilde{N})}.$ Then
\begin{eqnarray}\label{part2} \E\left[F\int_0^T\int_{\R_0}\Psi(t,z)\tilde{N}(dt,dz)\right]=\E\left[\int_0^T\int_{\R_0}\Psi(t,z)D_{t,z}F\mu(dz)dt\right].
\end{eqnarray}
\section{Stochastic maximum principle for mean-field type optimal control- Malliavin calculus approach}
In this  section,  we derive  the maximum principle for the mean
field optimal control problem  of minimizing (\ref{performance})
over $v(\cdot)\in {\mathscr{U}}_{ad}$ subject to (\ref{A}) and (\ref{observation}).

%
%
%
%
%
%

We make some assumptions on the   coefficients  $b\,, \sigma :
[0,T]\times\R\times U\times\Omega\rightarrow\R $
  and
$\gamma: [0,T]\times\R\times U\times
{\R_0}\times\Omega\rightarrow\R$:


\medskip
\begin{description}
\item{\bf  (A1)}\  The functions $b$, $\sigma$ and   $\gamma$ are
almost surely continuous
 with respect to their variables $t$, $x$, $v$.
 For any $t\in [0,T]$, the functions $b$ and $\sigma$ are continuously differentiable with respect to
 $x$ and $v$ with uniformly  bounded
 derivatives $b_x$, $b_v$, $\sigma_x$ and $\sigma_v$.
 \begin{equation}
 \sup_{0\le t\le T, x\in \R, v\in U\,,  \omega\in \Omega}\left[
 |b_x(t,x,v, \omega)|+|b_v(t,x,v, \omega)|+|\sigma _x(t,x,v, \omega)|+|\sigma _v(t,x,v,
 \omega)|\right]<\infty\,.
  \end{equation}
The function $\gamma$ is continuously differentiable in $(x,v)$ and
there is a constant $C$ such that
\[
\sup_{0\le t\le T,    \omega\in \Omega} \left( \int_{\R_0}|\gamma(t,x,v,z, \omega)|^2\mu(dz)\right)^{\frac{1}{2}} \le
C(1+|x|+|v|)\,.
\]
Moreover, we assume that $\displaystyle
\int_{\R_0}|\gamma_x(t,x,v,z)|^2\mu(dz)$ and  $\displaystyle
\int_{\R_0}|\gamma_v(t,x,v,z)|^2\mu(dz)$ are  continuous with
respect to $(x,v)$ and  uniformly bounded for $0\le t\le T, x\in \R,
v\in U$.
\item{\bf  (A2)}\  The function $h(t,x)$ is almost surely
continuous on $t\in[0, T]$ and $x\in \R$.   For any $t\in [0,T] $,
the function $h: [0,T]\times \R\times\Omega\rightarrow\R$ is
continuously differentiable with respect to $x$ and
\[
\sup_{t\in [0, T], x\in \R\,, \omega\in \Omega}\left[
|h(t,x)|+h_x(t, x)|\right]<\infty\,.
\]
For any $x$,
 $  h(t,x,\omega)$ is an $\mathscr{F}_t$-adapted process.%
 \end{description}

%
%
 The state process $(x^v(t)\,, 0\le t\le T)$ is not observable itself, but is observed partially and
 corrupted with noise
 so that we have  $(Y(t)\,, 0\le t\le T)$ defined by \eref{observation} available.
 Our control will be based on the observation of the process $Y$ up time instant
 $t$.

 The intrinsic difficulty arising from the fact that the control $v$ depends on the observation $Y$, which itself is dependent on
  the control $v$. The approach via Girsanov transformation  offers a way to overcome this difficulty.  Let
\begin{eqnarray}\label{rho}
\rho_0^v(t) 
=\exp\left\{-\int_0^t{h(s,x (s))}dW_2(s)-\frac{1}{2}\int_0^t{h^2(s,x
(s))}ds\right\}\,.
\end{eqnarray}
Define
\begin{equation}
\frac{d\P^v}{d\P_0}=\rho_0^v(T)  =\exp\left\{-\int_0^T {h(s,x
(s))}dW_2(s)-\frac{1}{2}\int_0^T {h^2(s,x (s))}ds\right\}\,.
\end{equation}
Then from the Girsanov theorem and the Kallinapur-Striebel formula,
we know that under the probability measure $\P^v$,  $(Y(t), 0\le
t\le T)$ is a Brownian motion, independent of $W_1$ and $N$. From
now on we shall use this probability measure $\P^v$.
Now we denote by $\P$ the probability measure, under which
$(W_1(t), Y(t), 0\le
t\le T)$ is a two dimensional Brownian motion  and $N$ is a Poisson random measure
independent of $W_1$ and $Y$.
The original
probability measure can be represented as
\begin{eqnarray}
\frac {d\P_0}{d\P} &=&\frac{1}{\rho_0^v(T)}
=\exp\left\{\int_0^T{h(s,x (s))}dW_2(s)+\frac{1}{2}\int_0^T{h^2(s,x
(s))}ds\right\}=\rho^v(T)\,,
\end{eqnarray}
where
\begin{eqnarray}
\rho^v(t) &=&   \exp\left\{\int_0^t{h(s,x
(s))}dY(s)-\frac{1}{2}\int_0^t{h^2(s,x (s))}ds\right\} \,.
\end{eqnarray}
It is easy to see that
\begin{eqnarray}\label{rhosolution}\left\{%
\begin{array}{lll}
    d\rho^v(t)&=&{h(t,x^v(t))}\rho^v(t)dY(s),\\
\rho^v(0)&=&1.
\end{array}%
\right.
\end{eqnarray}
%
%
With the new probability measure $ \P $ and denoting the expectation
with respect to $\P $ by $\E $, the performance functional
(\ref{performance}) can be rewritten as
\begin{eqnarray}
J(v(\cdot))&=& \E \left[\rho^v (T) \int_0^T  l(t, x^v(t), \E [\rho^v (T)f(x^v(t))], v(t) )dt\right. \nonumber\\
&&+\rho^v(T)\phi(x^v(T),\E [\rho^v(T)g(x^v(T))] )\Bigr]\nonumber\\
&=& \E \left[\int_0^T \rho^v(t)l(t, x^v(t), \E [\rho^v(t)f(x^v(t))], v(t) )dt\right. \nonumber\\
&&+\rho^v(T)\phi(x^v(T),\E [\rho^v(T)g(x^v(T))]
)\Bigr]\label{performanceP}
\end{eqnarray}
by the martingale property of $\rho^v(t)$.   Therefore,  the
problem (MFC) is equivalent to minimizing (\ref{performanceP}) over
${\mathscr{U}}_{ad}$ subject to (\ref{A}),
where in the definition of ${\mathscr{U}}_{ad}$  the observation
$(Y(t), 0\le t\le T)$ is a Brownian motion, independent of
$W_1$ and $N$  and  $\rho^v(t)$ is given by  (\ref{rhosolution}).

Let $\D_{1,2}$ denote the set of all random variables which are
Malliavin differentiable with respect to all of $W_1(\cdot)$,
$Y(\cdot)$, and $\tilde{N}(\cdot,\cdot).$

Furthermore, let us introduce some notations.
\begin{eqnarray}\label{notation}
a_\theta(t)&=&a_\theta(t,x(t),u(t))  \quad  {\mbox{for }}a=b,\sigma
{\mbox{ and }} \theta=x,v.\nonumber\\ 
h_x(t)&=&h_x(t,x(t)),\  l(u(t))=l(t,x(t),\E_0 [f( x(t))],u(t))\nonumber\\
l_\theta(t)&=& l_\theta(t,x(t),\E_0  [f(x(t))],u(t))  \quad  {\mbox{for }}
\theta=x,y,v.\nonumber\\
\phi&=&\phi(x(T),\E_0 [ g(x(T))]),\quad  \phi_\theta=\phi_\theta(x(T),\E_0 [ g(x(T))]) \quad  {\mbox{for }} \theta=x,y.\nonumber\\
%
 {\mathscr L}^2_{\mathscr{F}}(0,T)
 &=&\Bigl\{\phi(t,\omega) \mbox{ is an
 $\R$-valued progressively measurable process such that} \nonumber\\
&&\qquad \E\Bigl(\int_0^T|\phi(t)|^2dt\Bigr)<\infty\Bigr\}.\nonumber\\
{\mathscr M}^2_{\mathscr{F}}(0,T;\R)
&=&\Bigl\{\phi(t,z,\omega) \mbox{ is an
 $\R$-valued progressively measurable process} \nonumber\\
&&\qquad \mbox{  such that }
\E\Bigl(\int_0^T\int_{\R_0}|\phi(t,z)|^2\mu(dz)dt\Bigr)<\infty\Bigr\}.\nonumber\\
\L_{12}(\R)&=&\Bigl\{ \hbox{$F(t,\omega)$ is an  $\R$-valued progressively measurable process such that}:\nonumber\\
&& \hbox{ for almost everywhere $0\leq t\leq T$, \ $F(t,\cdot)\in \D_{1,2}$ and
\ \  }\nonumber\\
&&\qquad\hbox{$\|F\|_{1,2}^2:=\E
\Bigl(\int_0^T|F(t,\omega)|^2dt$} \hbox{ $+\int_0^T\int_0^T|
D_s^{(Y)}F(t,\omega)|^2ds dt$} \nonumber\\
&&\qquad  \hbox{$ +\int_0^T\int_0^T\int_{\R_0}|
D_{s,z}F(t,\omega)|^2\mu(dz)ds dt\Bigr)<\infty  $\Big\} }.
\end{eqnarray}

Let $u(\cdot), v(\cdot)\in {\mathscr{U}}_{ad}$ be given and fixed
such that $u(\cdot)+v(\cdot)\in {\mathscr{U}}_{ad}.$   For any
$0\leq \varepsilon\leq 1,$  we take the variational control
$u^\varepsilon(\cdot)=u(\cdot)+\varepsilon v(\cdot).$  Because
${\mathscr{U}}_{ad}$ is convex, $u^\varepsilon(\cdot)$ belongs to
${\mathscr{U}}_{ad}$.  Denote by $x^\varepsilon(\cdot)$ and
$\rho^\varepsilon(\cdot)$ the solutions  of (\ref{A}) and
(\ref{rhosolution}) corresponding to  the control
$u^\varepsilon(\cdot).$  When $\vare=0$, denote $x=x(\cdot)$ and
$\rho=\rho(\cdot)$.   To obtain the maximum principle for the
optimal control problem of minimizing (\ref{performanceP})
 over $v(\cdot)\in {\mathscr{U}}_{ad}$,  we use
 $\displaystyle \frac{d}{d\vare}\big|_{\vare=0} J(v^\vare)=0$.
To compute $\frac{d}{d\vare}\big|_{\vare=0} J(v^\vare)$ we need to compute
$\frac{d}{d\vare}\big|_{\vare=0}x^\varepsilon $ and
$\frac{d}{d\vare}\big|_{\vare=0}\rho^\varepsilon $, which are given
by the following lemma \ref{Lemma2}.  First,  we need
\begin{Lemma}\label{lemma1} Under  assumptions (A1) and (A2), for any $v(\cdot)\in {\mathscr{U}}_{ad},$ there is a constant $C$ such that
\begin{eqnarray*}
\begin{array}{lll}
& \sup\limits_{0\leq t\leq T}\E[x^v(t)]^2\leq C(1+\sup\limits_{0\leq t\leq T}\E v^2(t)),\quad \sup\limits_{0\leq t\leq T}\E[\rho^v(t)]^2\leq C,\\
&\sup\limits_{0\leq t\leq T}\E|x^{\varepsilon }(t)-x(t)|^2\leq
C\varepsilon^2,\quad  \sup\limits_{0\leq t\leq
T}\E|\rho^{\varepsilon}(t)-\rho(t)|^2\leq C\varepsilon^2.
\end{array}
\end{eqnarray*}
\end{Lemma}
\begin{proof} This is a directly application of
Burkholder-Davis-Gundy inequality.
\end{proof}

Consider the following linear stochastic differential  equations
(which will be the equations satisfied by
$\frac{d}{d\vare}\big|_{\vare=0}x^\varepsilon $ and
$\frac{d}{d\vare}\big|_{\vare=0}\rho^\varepsilon $).
\begin{eqnarray}\label{x1}\left\{%
\begin{array}{lll}
    dx_1(t)&=& \Bigl(b_x(t)x_1(t)+b_v(t)v(t)\Bigr)dt+\Bigl(\sigma_x(t)x_1(t)+\sigma_v(t)v(t)\Bigr)dW_1(t)\\
    &&+\displaystyle\int_{\R_0}[\gamma_x(t,x(t-),u,z)x_1(t-)+\gamma_v(t,x(t-),u,z)v(t)]\tilde{N}(dt,dz),\\
x_1(0)&=&0
\end{array}%
\right.
\end{eqnarray}
and
\begin{eqnarray}\label{rhosolution1}\left\{%
\begin{array}{lll}
    d\rho_1(t)&=& \Bigl(\rho_1(t) h(t,x(t))+\rho(t)h_x(t)x_1(t)\Bigr)dY(t),\\
\rho_1(0)&=&0\,,
\end{array}%
\right.
\end{eqnarray}
where  the notations $b_x(t)$ and so on are defined in (\ref{notation})
of this  section. In view of the boundedness of $b_x$, $b_v$,
$\sigma_x$, $\sigma_v$, $\gamma_x$, $\gamma_v$, $h$ and $h_x$,
(\ref{x1}) and (\ref{rhosolution1}) admit   unique solutions  $x_1(\cdot)\,,\rho_1(\cdot)
\in {\mathscr L}^2_{\mathscr{F}}(0,T;\R) $(See also \cite{BBP} and \cite{TangL}). Obviously,
$$\rho_1(t)=\rho(t)\Bigl(\int_0^th_x(s)x_1(s)dY(s)-\int_0^th_x(s)h(s,x(s))x_1(s)ds\Bigr),\quad 0\leq t\leq T.$$

By  Lemma 1 and Lemma 2 in \cite{Xiao}, we have the following lemma, which
states that $\frac{d}{d\vare}\big|_{\vare=0}x^\varepsilon=x_1(t) $
and $\frac{d}{d\vare}\big|_{\vare=0}\rho^\varepsilon=\rho_1(t) $.
\begin{Lemma}\label{Lemma2} Let the assumptions (A1) and (A2) hold. Then
\begin{eqnarray*}
\begin{array}{lll}
& \lim\limits_{\varepsilon\rightarrow 0}\sup\limits_{0\leq t\leq
T}\E\left|\displaystyle
\frac{x^\varepsilon(t)-x(t)}{\varepsilon}-x_1(t)
 \right| ^2 =0,\quad \lim\limits_{\varepsilon\rightarrow
0}\sup\limits_{0\leq t\leq T}\E\left|\displaystyle
\frac{\rho^\varepsilon(t)-\rho(t)}{\varepsilon}-\rho_1(t) \right|
^2=0.
\end{array}
\end{eqnarray*}
\end{Lemma}

The following assumptions are  needed to obtain the  maximum
principle.

\begin{description}
\item{\bf (A3)}\  The functions  $l:[0,T]\times\R\times\R\times U\times\Omega\rightarrow\R$ and
$\phi:\R\times\R\times\Omega\rightarrow\R$ are almost surely continuously differentiable
 with respect to  $(t,x,y,v)\in [0,T]\times \R\times\R\times U$  and $(x,y)\in \R\times U$,
respectively, and satisfying
\begin{eqnarray*}&&\E \left[\int_0^T \rho^v(t)|l(t, x^v(t), \E [\rho^v(t)f(x^v(t))], v(t) )|dt\right. \nonumber\\
&&\quad\quad+\rho^v(T)|\phi(x^v(T),\E [\rho^v(T)g(x^v(T))])|\Bigr]<\infty.
\end{eqnarray*}

$\phi$ is almost surely twice continuously  differentiable with respect to $x$ with first and second order bounded derivatives.
$f:\R\to\R$ and $g:\R\to\R$ are both twice continuously  differentiable with first and second order bounded derivatives .
\item{\bf (A4)}\  (i) For any $t,$ $\tau$, such that $t+\tau\in [0,T]$, and
bounded $\mathscr{F}_t^Y$-measurable random variable $\beta$, we
formulate the control process $v(s)\in U,$ with
$$v(s)=\beta I_{[t,t+\tau]}(s),\quad s\in[0,T],$$
where $I_{[t,t+\tau]}(s)$ is the indicator function on the set
$[t,t+\tau].$

(ii) For any $v(s) \in \mathscr{F}_s^Y$ with $v(s)$ bounded, $s\in
[0,T],$ there is an $\delta >0$ such that $u(\cdot)+\varepsilon
v(\cdot)\in \mathscr{U}_{ad}$ for $\varepsilon\in (-\delta,\delta).$
\end{description}

To describe the maximum principle  we define the following  adjoint
processes $q(\cdot)$, $k(\cdot)$ and $r(\cdot,\cdot)$ as follows. Let
\begin{eqnarray}
G(t,s)&=& {\rm
exp}\left(\int_t^s\Bigl[b_x(r)-\frac{1}{2}\sigma_x^2(r)\Bigr]dr
+\int_t^s \sigma_x(r)dW_1(r)\right.\nonumber\\
&&+\int_t^s\int_{\R_0}{\rm ln}\left(1+\gamma_x(r,x(r),u,z)\right)\tilde{N}(dr,dz)       \nonumber\\
&&+\left. \int_t^s\int_{\R_0}\Bigl[{\rm
ln}\left(1+\gamma_x(r,x(r),u,z)\right)-\gamma_x(r,x(r),u,z)\Bigr]\mu(dz)dr
\right); \ s>t\,;\label{G}\nonumber\\
\Sigma(t)&=&\rho(T)\Bigl(\phi_x+g^\prime(x(T))\E_0 [\phi_y]\Bigr)\nonumber\\
&&+\int_t^T \rho(s)\Bigl(l_x(s)+f^\prime(x(s))\E_0[l_y(s)]\Bigr)ds\,;\nonumber\\
\Pi(t)&=&\rho(T)\Bigl(\phi+g(x(T))\E_0 [\phi_y]\Bigr)\nonumber\\
&&+\int_t^T
\rho(s)\Bigl(l(u(s))+f(x(s))\E_0 [l_y(s)]\Bigr)ds\,; \label{e.Pi} \\
\Lambda(t)&=&\rho(T)\Bigl(\phi+g(x(T))\E_0[\phi_y]\Bigr)h(t,x(t))\nonumber\\
&&+\int_t^T\rho(s)\Bigl(l(u(s))+f(x(s))\E_0[l_y(s)]\Bigr)h(t,x(t))ds;\label{Lambda}\\
H_x(t)&=&\Sigma(t)b_x(t)+\sigma_x(t) D_t^{({W_1})}\Sigma(t)+{h_x(t)}\Bigl(D_t^{({Y})}\Pi(t)-\Lambda(t)\Bigr)\nonumber\\
&&+\int_{\R_0}\gamma_x(t,x(t-),u,z)D_{t,z}\Sigma(t)\mu(dz);\nonumber\\
\Theta(t,s)&=&H_x(s)G(t,s)\,.\nonumber
\end{eqnarray}
Finally we denote
\begin{eqnarray}
q(t)&:=&\Sigma(t)+\int_t^T \Theta(t,s) ds,\nonumber\\
k(t)&:=& D_t^{(W_1)} q(t),\nonumber\\
{r}(t,z)&:=&  D_{t,z}q(t).\label{pt}
\end{eqnarray}

Now we state our main theorem of this section.
\begin{theorem} Let the assumptions (A1), (A2), (A3) and (A4) hold. Assume that $u(\cdot)$
is a local minimum for $J(v(\cdot))$, in the sense that for all
bounded $v(\cdot)\in \mathscr{U}_{ad},$ there exists a  $\delta >0$
such that $u(\cdot)+\varepsilon v(\cdot)\in \mathscr{U}_{ad}$ for
any $\varepsilon\in (-\delta,\delta)$ and
$${\cal J}(\varepsilon)=J(u(\cdot)+\varepsilon v(\cdot)),\quad \varepsilon\in (-\delta,\delta),$$
attains its minimum at $\varepsilon=0.$  Assume that
$\rho(t)$, $l(u(t)),$ $l_x(t)$ and $\Theta(t,s)$ are in
$\mathbb{L}_{1,2}(\R)$ for all $0\leq t\leq s\leq T.$
   Then we have
$$\E \left[H_v(t,x(t),\E_0 [f(x(t))],u(t); q(t),k(t),r(t,\cdot))\left|\right. \mathscr{F}_t^Y\right]=0,$$
where $H_v$ is defined by
\begin{eqnarray*}&&H_v(t,x,y,v; q,k, {r})=b_v(t,x,v)q+\sigma_v(t,x,v)k\\
&&\quad\quad+\int_{\R_0}r(t,z)\gamma_v(t,x,v,z)\mu(dz)+\rho(t)l_v(t,x,y,v).
 \end{eqnarray*}
 \end{theorem}

\begin{proof}  If $u(\cdot)$ is a local minimum for $J(v(\cdot)),$
then $ \frac{d}{d\varepsilon}{\cal
J}(\varepsilon)|_{\varepsilon=0}=0$.
Since \begin{eqnarray*}
\frac{d}{d\varepsilon}{\cal J}(\varepsilon)|_{\varepsilon=0}
&=& \lim\limits_{\varepsilon\rightarrow 0}
\frac{J(u(\cdot)+\varepsilon v(\cdot))-J(u(\cdot))}{\varepsilon} \nonumber\\
&=& \lim\limits_{\varepsilon\rightarrow 0} \frac{1}{\varepsilon} \E\left\{\int_0^T
\Bigl[\left(\rho^{u+\varepsilon v}(t)-\rho(t)\right)l(t,x(t),\E[\rho(t)f(x(t))],u)
\right.\nonumber\\
&&+ \rho^{u+\varepsilon v}(t)\Bigl(l(t,x^{u+\varepsilon v}(t),\E[\rho^{u+\varepsilon v}
(t)f(x^{u+\varepsilon v}(t))],u+\varepsilon v )\nonumber\\
&&-l(t,x(t),\E[\rho(t)f(x(t))],u)\Bigr)\Bigr]dt\nonumber\\
&&+\left(\rho^{u+\varepsilon v}(T)-\rho(T)\right)\phi(x(T),\E[\rho(T)g(x(T))])\nonumber\\
&&\Bigl.+\rho^{u+\varepsilon v}(T)\Bigl(\phi(x^{u+\varepsilon v}(T),\E[\rho^{u+\varepsilon
v}(T)g(x^{u+\varepsilon v}(T))])\nonumber\\
&&-\phi(x(T),\E[\rho(T)g(x(T))])\Bigr)\Bigr\},\nonumber\\
\end{eqnarray*}
it follows from  Lemma 3.2 that
\begin{eqnarray*}
\frac{d}{d\varepsilon}{\cal J}(\varepsilon)|_{\varepsilon=0}
&=&\E  \Bigl[\int_0^T\rho_1(t) l(t,x(t),\E_0 [f(x(t))], u)dt\Bigr]\nonumber\\
&&+\E  \int_0^T\Bigl[l_x(t,x(t),\E_0 [f(x(t))],u)\rho(t)x_1(t)\nonumber\\
&&+\rho(t)l_y(t,x(t),\E_0 [f(x(t))],u)
\E \Bigl(f^\prime(x(t)\rho(t)x_1(t)+\rho_1(t)f(x(t))\Bigr)\Bigr]dt\nonumber\\
&&+\E  \int_0^T l_v(t,x(t),\E_0 [f(x(t))],u)\rho(t)v(t)dt\nonumber\\
&&+\E  \Bigl[\phi(x(T),\E_0 [g(x(T))])\rho_1(T)+\phi_x(x(T),\E_0 [g(x(T))])\rho(T)x_1(T)\nonumber\\
&&+\rho(T)\phi_y(x(T),\E_0 [g(x(T))])\E \Bigl(g^\prime(x(T))\rho(T)x_1(T)+\rho_1(T)g(x(T))\Bigr)\Bigr].\nonumber\\
\end{eqnarray*}
For the convenience of computation, we may adjust the order of the terms in the right side of the above equation such that
\begin{eqnarray}
\label{J}
\frac{d}{d\varepsilon}{\cal
J}(\varepsilon)|_{\varepsilon=0}&=&\E  \Bigl[\phi(x(T),\E_0 [g(x(T))])\rho_1(T)+\int_0^T\rho_1(t) l(t,x(t),\E_0 [f(x(t))], u)dt\Bigr]\nonumber\\
&&+\E\Bigl[\E_0 \bigl[\phi_y(x(T),\E_0 [g(x(T))])\bigr]
\rho_1(T)g(x(T))\Bigr]\nonumber\\
&&+\E \int_0^T \E_0\bigl[l_y(t,x(t),\E_0 [f(x(t))],u)\bigr]
\rho_1(t)f(x(t))dt\nonumber\\
&&+\E  \Bigl[\phi_x(x(T),\E_0 [g(x(T))])\rho(T)x_1(T)\nonumber\\
&&+\E_0\bigl[\phi_y(x(T),\E_0 [g(x(T))])\bigr]g^\prime(x(T))\rho(T)x_1(T)\Bigr]\nonumber\\
&&+\E
\Bigl[\int_0^Tl_x(t,x(t),\E_0 [f(x(t))],u)\rho(t)x_1(t)dt\nonumber\\
&&+\int_0^T\E_0 \bigl[l_y(t,x(t),\E_0 [f(x(t))],u)\bigr]f^\prime(x(t))\rho(t)x_1(t)dt\Bigr]\nonumber\\
&&+\E  \int_0^T l_v(t,x(t),\E_0 [f(x(t))],u)\rho(t)v(t)dt\nonumber\\
&=:&I_1+I_2+I_3+I_4+I_5+I_6+I_7+I_8+I_9\,,
\end{eqnarray}
where in the forth identity we have used  $\EE
(A\EE(B))=\EE(B \, \EE(A))$.  Since $\phi,$ $\phi_x$, $\E_0 [\phi_y]g(x(T))$ and $\E_0 [\phi_y]g^\prime(x(T))\in \mathbb{D}_{1,2}$, $l(u(t)),$
$l_x(t),$ $\E_0 [l_y(t)]f(x(t))$, $\E_0 [l_y(t)]f^\prime(x(t))$ and $\Theta(t,s)$ are in
$\mathbb{L}_{1,2}(\R)$ for all $0\leq t\leq s\leq T,$  according to
(\ref{part}) and (\ref{part2}), we have
\begin{eqnarray}\label{phigamma}
I_1&=&\E \Bigl(\phi(x(T),\E_0[g(x(T))])\rho(T)\Bigl[\int_0^T h_x(t)x_1(t)dY(t)-\int_0^T h_x(t)h(t,x(t))x_1(t)dt\Bigr]\Bigr)\nonumber\\
&=&\E \int_0^T { h_x(t) x_1(t)}\Bigl[D_t^{(Y)}\Bigl(\rho(T)\phi(x(T),\E_0[g(x(T))])\Bigr)\nonumber\\
&&-\rho(T)\phi(x(T),\E_0[g(x(T))])h(t,x(t))\Bigr]dt
\end{eqnarray}
and 
\begin{eqnarray}\label{Gammal}
I_2&=&\E\int_0^Tl(t,x(t),\E_0[f(x(t))], u)\rho(t)\Bigl[\int_0^t h_x(s)x_1(s)dY(s)-\int_0^t h_x(s)h(s,x(s))x_1(s)ds\Bigr]  \nonumber \\
& =&\E\int_0^T
{h_x(t) x_1(t)}\left\{\int_t^T\Bigl[D_t^{(Y)}\Bigl(\rho(s)l(u(s))\Bigr)-\rho(s)l(u(s))h(t,x(t))\Bigr]ds\right\}dt.
\end{eqnarray}
Note that, in deriving the last identity   in (\ref{Gammal}), we
have used the Fubini  theorem.   Similarly,
\begin{eqnarray}\label{Egammax}
I_3&=&\E \Bigl\{g(x(T))\E_0\Bigl(\phi_y(x(T),\E_0 [ g(x(T))])\Bigr)\rho(T)\Bigl[\int_0^T h_x(t)x_1(t)dY(t)\nonumber \\
&&-\int_0^T h_x(t)h(t,x(t))x_1(t)dt\Bigr]\Bigr\}\nonumber\\
&=&\E \int_0^T{ h_x(t)
x_1(t)}\Bigl[D_t^{(Y)}\Bigl(\rho(T)g(x(T))\E_0[\phi_y]\Bigr)-\rho(T)g(x(T))\E_0[\phi_y]h(t,x(t))\Bigr]dt
\end{eqnarray}
and
\begin{eqnarray}\label{Ely}
I_4&=&\E\int_0^T f(x(t))\E_0[l_y(t)]\rho(t) \Bigl[\int_0^t h_x(s)x_1(s)dY(s)-\int_0^t h_x(s)h(s,x(s))x_1(s)ds\Bigr]dt  \nonumber \\
&=&\E\int_0^T {h_x(t) x_1(t)}\int_t^T
\Bigl[D_t^{(Y)}\Bigl(\rho(s)f(x(s))\E_0[l_y(s)]\Bigr)\nonumber \\
&&-\rho(s)f(x(s))\E_0[l_y(s)]h(t,x(t))\Bigr]dsdt.
\end{eqnarray}
Then   from (\ref{phigamma}), (\ref{Gammal}), (\ref{Egammax}),
(\ref{Ely}), \eref{e.Pi} (the definition of $\Pi$) and \eref{Lambda} (the definition of $\Lambda$) it follows that
\begin{eqnarray}\label{phigammatildeW}
I_1+I_2+I_3+I_4
&=&\E\int_0^T\Biggl\{D_t^{(Y)}\Bigl(\rho(T)\phi\Bigr)+D_t^{(Y)}\Bigl(\rho(T)g(x(T))\E_0[\phi_y]\Bigr)\nonumber\\
&&+\int_t^TD_t^{(Y)}\Bigl(\rho(s)l(u(s))\Bigr)ds+\int_t^T  D_t^{(Y)}\Bigl(\rho(s)f(x(s))\E_0[l_y(s)]\Bigr)ds\nonumber\\
&&-\Bigl(\phi+g(x(T))\E_0[\phi_y]\Bigr)\rho(T)h(t,x(t))\nonumber\\
&&-\int_t^T\Bigl(l(u(s))+f(x(s))\E_0 [l_y(s)]\Bigr)\rho(s)h(t,x(t))ds\Biggr\}{h_x(t) x_1(t)}dt\nonumber\\
&=&\E\int_0^T  {h_x(t) x_1(t)}\Bigl( D_t^{(Y)}\Pi(t)-\Lambda(t)\Bigr)dt.
\end{eqnarray}
Similarly, according to (\ref{part}) and (\ref{part2}), we have
\begin{eqnarray}{\label{phix}}
I_5+I_6
&=&\E \left\{\rho(T)\Bigl(\phi_x+ g^\prime(x(T))\E_0[\phi_y]\Bigr)\left[\int_0^T\Bigl(b_x(t) x_1(t)+b_v(t) v(t)\Bigr)dt \right.\right.\nonumber\\
 &&+\int_0^T\Bigl({\sigma}_x(t) x_1(t)+{\sigma}_v(t) v(t)\Bigr)d {W_1}(t)\nonumber\\
&&+\left.\left.\int_0^T\int_{\R_0}\Bigl[\gamma_x(t,x(t-),u,z)x_1(t-)+\gamma_v(t,x(t-),u,z)v(t)\Bigr]\tilde{N}(dt,dz)\right]\right\}\nonumber\\
&=&\E\int_0^T\left\{\rho(T)\Bigl(\phi_x+ g^\prime(x(T))\E_0[\phi_y]\Bigr)\Bigl(b_x(t) x_1(t)+b_v(t) v(t)\Bigr)\right. \nonumber\\
&&+ \Bigl({\sigma}_x(t) x_1(t)+{\sigma}_v(t) v(t)\Bigr)D_t^{(W_1)}\rho(T)\Bigl(\phi_x+ g^\prime(x(T)) \E_0[\phi_y]\Bigr)\nonumber\\
&&+\int_{\R_0}\Bigl[\gamma_x(t,x(t-),u,z)x_1(t-)+\gamma_v(t,x(t-),u,z)v(t)\Bigr]\nonumber\\
&&\quad\quad\left.\left.D_{t,z}\rho(T)\Bigl(\phi_x+ g^\prime(x(T)) \E_0[\phi_y]\Bigr)\mu(dz)\right]\right\}dt
\end{eqnarray}
and
\begin{eqnarray*}
I_7+I_8&=&\E\int_0^T\rho(t)\Bigl(l_x(t)+f^\prime(x(t))\E_0[l_y(t)]\Bigr)\left\{\int_0^t\bigl(b_x(s)x_1(s)+b_v(s)v(s)\bigr)ds\right.\nonumber\\
&&+\int_0^t\left({\sigma}_x(s) x_1(s)+{\sigma}_v(s) v(s)\right)dW_1(s)\nonumber\\
&&+\left.\int_0^t\int_{\R_0}\Bigl[\gamma_x(s,x(s-),u,z)x_1(s-)+\gamma_v(s,x(s-),u,z)v(s)\Bigr]\tilde{N}(ds,dz)\right\}\nonumber\\
&=&\E\int_0^T\int_0^t\left\{\rho(t)\Bigl(l_x(t)+f^\prime(x(t))\E_0[l_y(t)]\Bigr)\bigl(b_x(s)x_1(s)+b_v(s)v(s)\bigr)\Bigr.\right. \nonumber\\
&&+\left({\sigma}_x(s) x_1(s)+{\sigma}_v(s) v(s)\right)D_s^{({W_1})}\rho(t)\Bigl(l_x(t)+f^\prime(x(t))\E_0[l_y(t)]\Bigr)\nonumber\\
&&+\int_{\R_0}\Bigl[\gamma_x(s,x(s-),u,z)x_1(s-)+\gamma_v(s,x(s-),u,z)v(s)\Bigr]\nonumber\\
&&\quad\quad\left.D_{s,z}\rho(t)\Bigl(l_x(t)+f^\prime(x(t))\E_0[l_y(t)]\Bigr)\mu(dz)\right\}dsdt.\nonumber\\
\end{eqnarray*}
By the Fubini theorem, we have
\begin{eqnarray}{\label{lx}}
I_7+I_8&=&\E\int_0^T\int_s^T\left\{\rho(t)\Bigl(l_x(t)+f^\prime(x(t))\E_0 [l_y(t)]\Bigr)\bigl(b_x(s)x_1(s)+b_v(s)v(s)\bigr)\Bigr.\right. \nonumber\\
&&+\left({\sigma}_x(s) x_1(s)+{\sigma}_v(s) v(s)\right)D_s^{({W_1})}\rho(t)\Bigl(l_x(t)+f^\prime(x(t))\E_0 [l_y(t)]\Bigr)\nonumber\\
&&+\int_{\R_0}\Bigl[\gamma_x(s,x(s-),u,z)x_1(s-)+\gamma_v(s,x(s-),u,z)v(s)\Bigr]\nonumber\\
&&\quad\quad\left.D_{s,z}\rho(t)\Bigl(l_x(t)+f^\prime(x(t))\E_0[l_y(t)]\Bigr)\mu(dz)\right\}dtds\nonumber\\
&=&\E\int_0^T\left\{\int_t^T\rho(s)\Bigl(l_x(s)+f^\prime(x(s))\E_0[l_y(s)]\Bigr)ds\bigl(b_x(t)x_1(t)+b_v(t)v(t)\bigr)\Bigr.\right. \nonumber\\
&&+\bigl({\sigma}_x(t) x_1(t)+{\sigma}_v(t) v(t)\bigr)\int_t^TD_t^{(W_1)}\rho(s)\Bigl(l_x(s)+f^\prime(x(s))\E_0[l_y(s)]\Bigr)ds\nonumber\\
&&+\int_t^T\int_{\R_0}\Bigl[\gamma_x(t,x(t-),u,z)x_1(t-)+\gamma_v(t,x(t-),u,z)v(t)\Bigr]\nonumber\\
&&\quad\quad\Bigl.D_{t,z}\rho(s)\Bigl(l_x(s)+f^\prime(x(s))\E_0[l_y(s)]\Bigr)\mu(dz)ds\Bigr\}dt.
\end{eqnarray}
Then it follows from (\ref{phix}) and (\ref{lx}) that
\begin{eqnarray}{\label{phixandlx}}
&&I_5+I_6+I_7+I_8\nonumber\\
&=&\E\int_0^T\Bigl\{\Sigma(t)\Bigl(b_x(t)x_1(t)+b_v(t)v(t)\Bigr)
+\Bigl({\sigma}_x(t) x_1(t)+{\sigma}_v(t) v(t)\Bigr)D_t^{({W_1})}\Sigma(t)\nonumber\\
&&+\int_{\R_0}\Bigl[\gamma_x(t,x(t-),u,z)x_1(t-)+\gamma_v(t,x(t-),u,z)v(t)\Bigr]\Bigl.D_{t,z}\Sigma(t)\mu(dz)\Bigr\}dt.
\end{eqnarray}
We insert  (\ref{phigammatildeW}) and (\ref{phixandlx}) into
(\ref{J}) to transform  the equation $\frac{d}{d\varepsilon}{\cal
J}(\varepsilon)|_{\varepsilon=0} =0$  to
\begin{eqnarray}\label{Jvarpsilon}
&&\E\int_0^T\Bigl[\Sigma(t)b_x(t)+\sigma_x(t) D_t^{({W_1})}\Sigma(t)+{h_x(t)}\Bigl(D_t^{({Y})}\Pi(t)-\Lambda(t)\Bigr)\nonumber\\
&&+\int_{\R_0}\gamma_x(t,x(t-),u,z)D_{t,z}\Sigma(t)\mu(dz)\Bigr]x_1(t)dt\nonumber\\
&&+\E\int_0^T\Bigl[\Sigma(t)b_v(t)+\sigma_v(t) D_t^{({W_1})}\Sigma(t)+\rho(t) l_v(t,x(t),\E_0[f(x(t))],u)\nonumber\\
&&+\int_{\R_0}\gamma_v(t,x(t-),u,z)D_{t,z}\Sigma(t)\mu(dz)\Bigr]v(t)dt=0\,.
\end{eqnarray}
To simplify the equation \eref{Jvarpsilon}, we take
\[
v(s)=\beta I_{(t,t+\tau]}(s)\,, \]
 where $\beta=\beta(\omega)$ is a
bounded $ \mathscr{F}^Y_t$-measurable random variables, $0\leq t\leq
t+\tau\leq T.$  It is easy to see from (\ref{x1}) that
\begin{eqnarray}\label{x1s}
x_1(s)=0 \ \ {\mbox{for\ }} 0\leq s\leq t.
\end{eqnarray}
Then (\ref{Jvarpsilon}) can be written as
\begin{eqnarray}
 {\cal J}_1(\tau)+{\cal J}_2(\tau)=0\label{e.J-tau}
\end{eqnarray}
with
\begin{eqnarray*}\label{J1varpsilon}
{\cal J}_1(\tau)&=&\E\int_{t}^T\Bigl[\Sigma(s)b_x(s)+\sigma_x(s) D_s^{({W_1})}\Sigma(s)+{h_x(s)}\Bigl(D_s^{({Y})}\Pi(s)-\Lambda(s)\Bigr)\nonumber\\
&&+\int_{\R_0}\gamma_x(s,x(s-),u,z)D_{s,z}\Sigma(s)\mu(dz)\Bigr]x_1(s)ds\nonumber\\
\end{eqnarray*}
and
\begin{eqnarray*}
{\cal J}_2(\tau)&=&\E\int_t^{t+\tau}\beta \Bigl[\Sigma(s)b_v(s)+\sigma_v(s) D_s^{({W_1})}\Sigma(s)+\rho(s) l_v(s,x(s),\E_0[f(x(s))],u)\nonumber\\
&&+\int_{\R_0}\gamma_v(s,x(s-),u,z)D_{s,z}\Sigma(s)\mu(dz)\Bigr]ds.
\end{eqnarray*}
Since \eref{e.J-tau} holds for all $\tau\in [0, T-t]$ we
differentiate it to obtain
\begin{eqnarray}
 \frac{d}{d\tau}\big|_{\tau=0}{\cal J}_1(\tau)+\frac{d}{d\tau}\big|_{\tau=0}{\cal J}_2(\tau)=0\,. \label{e.J}
\end{eqnarray}
First we compute $ \frac{d}{d\tau}\big|_{\tau=0}{\cal J}_1(\tau)$.
Note that with the special control $v(s)=\beta I_{(t,t+\tau]}(s)$,
we derive for $s\geq t+\tau$
\begin{eqnarray*}dx_1(s)&=& x_1(s)\Bigl(b_x(s)ds+\sigma_x(s)dW_1(s)
    +\int_{\R_0}\gamma_x(s,x(s-),u,z)\tilde{N}(ds,dz)\Bigr),
\end{eqnarray*}
Solving the above equation, we get
$$x_1(s)=x_1(t+\tau)G(t+\tau,s),\ \ \ s\geq t+\tau,$$
where
\begin{eqnarray*}
x_1(t+\tau)&=&\beta\int_t^{t+\tau}\Bigl(b_v(r) dr+\sigma_v(r)dW_1(r)
+\int_{\R_0}\gamma_v(r,x(r-),u,z)\tilde{N}(dr,dz)\Bigr)\\
&&+\int_t^{t+\tau}x_1(r)\Bigl(b_x(r) dr+\sigma_x(r)dW_1(r)
+\int_{\R_0}\gamma_x(r,x(r-),u,z)\tilde{N}(dr,dz)\Bigr).
\end{eqnarray*}
Then
\begin{eqnarray*}
\frac{d}{d\tau}{\cal J}_1(\tau)|_{\tau=0}&=& \frac{d}{d\tau}\E\left[\int_{t+\tau}^TH_x(s)x_1(t+\tau)G(t+\tau,s)ds\right]_{\tau=0}\nonumber\\
&=& \int_{t}^T \frac{d}{d\tau}\E \left[H_x(s)x_1(t+\tau)G(t+\tau,s)ds\right]_{\tau=0}ds\nonumber\\
&=& \int_{t}^T \frac{d}{d\tau}\E \left[x_1(t+\tau)\Theta(t,s)ds\right]_{\tau=0}ds\nonumber\\
&=& {\cal J}_{11}+ {\cal J}_{12},
\end{eqnarray*}
where
\begin{eqnarray*}
{\cal J}_{11}&=&\int_t^T\frac{d}{d\tau}\E\left\{\Theta(t,s)\int_t^{t+\tau}x_1(r)\Bigl(b_x(r) dr+\sigma_x(r)dW_1(r)\right.\\
&&\left.+\int_{\R_0}\gamma_x(r,x(r-),u,z)\tilde{N}(dr,dz)\Bigr)\right\}_{\tau=0}ds
\end{eqnarray*}
and
\begin{eqnarray*}
{\cal J}_{12}&=&\int_t^T\frac{d}{d\tau}\E\left\{\beta\Theta(t,s)\int_t^{t+\tau}\Bigl[b_v(r)dr+\sigma_v(r)dW_1(r)\Bigr.\right.\nonumber\\
&&\left.+\Bigl.\int_{{\R}_0}\gamma_v(r,x(r-),u,z)\tilde{N}(dr,dz)\Bigr]\right\}_{\tau=0}ds.
\end{eqnarray*}
According to (\ref{x1s}), (\ref{part}), (\ref{part2}) and the fact
that $x(t)=0$, it is not difficult to derive that
$${\cal J}_{11}=0$$
and
\begin{eqnarray}\label{J12}
{\cal J}_{12}&=&\E \int_t^T\beta\left(\Theta(t,s)b_v(t)+\sigma_v(t)D_t^{(W_1)}\Theta(t,s)\right.\nonumber\\
&&\left.+\int_{\R_0}\gamma_v(s,x(s-),u,z)D_{t,z}\Theta(t,s)\mu(dz)\right) ds.
\end{eqnarray}
Now we proceed to calculate the value of $\frac{d}{d\tau}{\cal
J}_2(\tau)|_{\tau=0}$.   As in the computation for
$\frac{d}{d\tau}{\cal J}_2(\tau)|_{\tau=0}$  we have
\begin{eqnarray}\label{dJ2}
\frac{d}{d\tau}{\cal J}_2(\tau)|_{\tau=0}&=&\E\left\{\beta \Bigl[\Sigma(t)b_v(t)+\sigma_v(t)D_t^{({W_1})}\Sigma(t)+\rho(t) l_v(t)\Bigr.\right.\nonumber\\
&&\left.\Bigl.+\int_{ \R_0}\gamma_v(t,x(t-),u,z)D_{t,z}\Sigma(t)\mu(dz)\Bigr]\right\}.
\end{eqnarray}
From (\ref{pt}), (\ref{J}), (\ref{J12}) and (\ref{dJ2}),  the equation
\eref{e.J} becomes
\begin{eqnarray*}
\E\left\{\beta \Bigl[b_v(t)q(t)+\sigma_v(t)k(t)+ \rho(t)l_v(t)+\int_{\R_0}\gamma_v(t,x(t-),u,z)r(t,z)\mu(dz)\Bigr]\right\}=0.
\end{eqnarray*}
Since the above equality holds for any bounded $\mathscr{F}_t^Y$-measurable $\beta$, we conclude that
$$0=\E\left[H_v(t,x(t),\E_0 [f(x(t))],u(t); q(t),k(t),r(t,\cdot))\left|\right. \mathscr{F}_t^Y\right].$$
The proof of the theorem is then completed.
\end{proof}

\medskip
\noindent{\bf An application to  linear-quadratic control problem}
We consider an economic quantity $x^v(\cdot),$ which can be
interpreted as cash-balance, wealth, and an intrinsic value process
in different fields of insurance, mathematical finance, and
mathematical economic, respectively. Suppose that $x^v(\cdot)$ is
governed by
\begin{eqnarray}\label{xnew}\left\{%
\begin{array}{lll}
    dx^v(t)&=& \Bigl(A(t)x^v(t)+B(t)v(t)\Bigr)dt+\Bigl(C(t)x^v(t)+D(t)v(t)\Bigr)dW_1(t)\\
&& +\displaystyle\int_{\R_0}\Bigl(F_t(z)x^v(t-)+G_t(z)v(t)\Bigr)\tilde{N}(dt,dz),\quad t\in[0,T],\\
x^v(0)&=&x_0\in \R, \\
\end{array}%
\right.
\end{eqnarray}
where $v(\cdot)$ is the control strategy of a policymaker, and $A(t)$, $B(t)$, $C(t)$, $D(t)$ , $F_t(z)$ and $G_t(z)$
are uniformly bounded $\mathscr{F}_t^Y$-adapted stochastic processes with value in $\R$. In fact, it is possible for the policymaker to partially observe $x(\cdot)$, due to the inaccuracies
in measurements, discreteness of account information, or possible delay in the actual payments. See, e.g., Huang, Wang, and Wu \cite{Huang},
Xiong and Zhou \cite{Xiong}, and {\O}ksendal and Sulem \cite{HP}. For this, we consider the following factor model:
\begin{eqnarray}\label{Zhnew}\left\{%
\begin{array}{lll}
    dY(t)&=&  \Bigl(\frac{1}{\beta}\alpha(t,x^v(t))-\frac{1}{2}\beta\Bigr)dt+dW_2(t),\\
Y(0)&=&0,
\end{array}%
\right.
\end{eqnarray}
where $x(\cdot)$ is the underlying factor which is partially oberved through the observation $Y(\cdot)$, $\beta>0$ is a constant, and
$\alpha$ satisfies an assumption similar to $h$ (see, e.g., Assumption (A2)). A typical example of $Y(\cdot)$ in reality is the logarithm of
the stock price $S(\cdot)$ related to $x(\cdot)$. Specifically, set $S(t)=s_0 e^{\beta Y(t)}$ with a constant $s_0>0.$ Obviously, the stock price $S(\cdot)$
is the information available to the policymaker. Moreover, it follows from It\^o's formula that
\begin{eqnarray*}\left\{%
\begin{array}{lll}
    dS(t)&=& S(t)\Bigl[\alpha(t,x^v(t))dt+\beta dW_2(t)\Bigr],\\
S(0)&=&s_0,
\end{array}%
\right.
\end{eqnarray*}
Note that the above factor model is inspired by those of Nagai and Peng \cite{Nagai} and
Xiong and Zhou \cite{Xiong}.

Assume that the objective of the policymaker is to minimize
\begin{eqnarray}\label{example}J(v(\cdot))&=&\frac{1}{2} \E_0 \left\{\int_0^T \Bigl[L(t)\Bigl( x^v(t)-\E_0 [x^v(t)]\Bigr)^2+\Bigl(v(t)-M(t)\Bigr)^2\Bigr]dt\right.\nonumber\\
&&\left.+N\Bigl(x^v(T)-\E_0 [x^v(T)]\Bigr)^2\right\},
\end{eqnarray}
subject to (\ref{xnew}) and (\ref{Zhnew}), where $M(t)\geq 0$, $L(t)\geq 0$ are uniformly bounded  deterministic functions with value in $\R$ and $M(t)$ is referred to as a dynamic benchmark. $N\geq 0$ is a constant.
Equation (\ref{example}) implies that the policymaker wants to not only prevent the control strategy from large deviation but also minimize the risk of
the economic quantity.

In what follows, we solve the linear-quadratic problem with the help of Theorem 2.1. It is easy to see from (\ref{xnew}) and (\ref{Zhnew}) that
$$b(t,x,v)=A(t)x+B(t)v,\quad \sigma(t,x,v)=C(t)x+D(t)v,$$
$$\gamma(t,x,v,z)=F_t(z)x+G_t(z)v,\quad h(t,x)=\frac{1}{\beta}\alpha(t,x)-\frac{1}{2}\beta.$$
As we know,
\begin{eqnarray*}
\rho^v(t) &=&   \exp\left\{\int_0^t{h(s,x^v
(s))}dY(s)-\frac{1}{2}\int_0^t{h^2(s,x^v (s))}ds\right\} \,.
\end{eqnarray*}
If $u$ is the optimal control, than we denote $\rho(t)=\rho^u(t),\ 0\leq t\leq T.$
The new adjoint processes are written as
\begin{eqnarray}\label{ptnew}
&q(t)=\Sigma(t)+\int_t^T H_x(s)G(t,s) ds,\nonumber\\
& k(t)= D_t^{(W_1)} q(t), \quad {r}(t,z)= D_{t,z}q(t),
\end{eqnarray}
with
\begin{eqnarray*}
\Sigma(t)&=&N\rho(T)\Bigl(x(T)-\E_0[x(T)]\Bigr)+\int_t^T L(s)\rho(s)\Bigl(x(s)-\E_0[x(s)]\Bigr)ds,
\end{eqnarray*}
where
\begin{eqnarray*}
G(t,s)&=& {\rm exp}\left(\int_t^s\Bigl[A(r)-\frac{1}{2}C^2(r)\Bigr]dr
+\int_t^s C(r)dW_1(r)\right.\nonumber\\
&&+\int_t^s\int_{ \R_0}{\rm ln}\left(1+F_r(z)\right)\tilde{N}(dr,dz)       \nonumber\\
&&+\left. \int_t^s\int_{\R_0}\Bigl[{\rm ln}\left(1+F_r(z)\right)-F_r(z)\Bigr]\mu(dz)dr   \right), \ s>t,
\end{eqnarray*}
\begin{eqnarray*}
H_x(t)&=&A(t)\Sigma(t)+C(t) D_t^{({W_1})}\Sigma(t)+\frac{1}{\beta}\alpha_x(t,x)D_t^{({Y})}\Pi(t)\nonumber\\
&&-\frac{1}{\beta}\alpha_x(t,x)\Lambda(t)+\int_{\R_0}F_t(z)D_{t,z}\Sigma(t)\mu(dz),
\end{eqnarray*}
and
\begin{eqnarray*}
\Pi(t)&=&\frac{1}{2}N\rho(T)\Bigl(x(T)-\E_0[x(T)]\Bigr)^2
+\frac{1}{2}\int_t^T\rho(s) \Bigl[L(s)\Bigl( x(s)\\
&&-\E_0 [x(s)]\Bigr)^2+\Bigl(u(s)-M(s)\Bigr)^2\Bigr]ds,\\
\Lambda(t)&=&\frac{1}{2}N\rho(T)\Bigl(x(T)-\E_0[x(T)]\Bigr)^2h(t,x(t))
\\
&&+\frac{1}{2}\int_t^T\rho(s) \Bigl[L(s)\Bigl( x(s)-\E_0 [x(s)]\Bigr)^2+\Bigl(u(s)-M(s)\Bigr)^2\Bigr]h(t,x(t))ds.
\end{eqnarray*}
According to Theorem 2.1 and (\ref{ptnew}), we have the following proposition.
\begin{proposition}  If $u(\cdot)$ is an optimal control strategy and $\rho(t)$, $\frac{1}{\beta}\alpha_x(t,x(t))G(t,s)\in \L_{1,2}(\R)$, $0\leq t\leq s\leq T$, then it is necessary to satisfy
\begin{eqnarray*}u(t)=M(t)-B(t)\E[q(t)|\mathscr{F}_t^Y]-D(t)\E[ D_t^{({W_1})}q(t)|\mathscr{F}_t^Y]-\E \Bigl[\int_{\R_0}G_t(z)D_{t,z}q(t)\mu(dz)|\mathscr{F}_t^Y\Bigr].
\end{eqnarray*}
\end{proposition}

\section{Maximum principle for jump-diffusion mean-field SDEs}
In this section, we study the mean field stochastic optimal control
problem to minimize \eref{performance}. However,  the   system is
given by a  nonlinear SDE of mean-field type (which is also called
McKean-Valasov equations) with jumps, namely, \eref{Anew}.  
The observation is as  \eref{observation} and we define the
admissible control as Definition \ref{d.adm}. 


As for the first problem treated in the previous section, we need to
deal with the problem of minimizing the performance functional
\eref{performanceP} subject to new state equation \eref{Anew} and the
observation  \eref{observation}.  The Radon-Nikodym derivative
$\rho^v$ is still given by \eref{rhosolution}. To obtain the
maximum principle for this problem, we make  the following
assumptions  in this section.

\begin{description}
\item{\bf (H1)}\  For any $t\in [0,T]$ and $z\in \R_0$, $b(t,x,y,v)$,  $\sigma(t,x,y,v)$
and $\gamma(t,x,y,v,z )$ are continuously differentiable functions of
$x, y$ and $v$ and  their derivatives $b_x$, $b_y$, $b_v$,
$\sigma_x$, $\sigma_y$, $\sigma_v$,
$\int_{\R_0}|\gamma_x(t,x,y,v,z)|^2\mu(dz)$,
$\int_{\R_0}|\gamma_y(t,x,y,v,z)|^2\mu(dz)$ and
$\int_{\R_0}|\gamma_v(t,x,y,v,z)|^2\mu(dz)$   are  uniformly
bounded. Suppose also that there is a constant $C>0$ such that
\begin{eqnarray}
|b(t,x,y,v)|^2+|\sigma(t,x,y,v)|^2+
\int_{\R_0}|\gamma(t,x,y,v,z)|^2\mu(dz)  \le
C(1+|x|^2+|y|^2+|v|^2)\,.
\end{eqnarray}


\item{\bf (H2)}\ For any $t\in [0,T],$ the function $h$ is continuously
differentiable with respect to $x$    and its derivative $h_x$ are
uniformly bounded.

\item{\bf (H3)} For any $t\in
[0,T]$, the functions  $l$ and $\phi$ are continuously differentiable
with respect to $(x,y,v) \in \R\times\R\times U$ and $(x,y) \in
 \R\times\R $, respectively. The derivatives of $l$ and $\phi$ are
uniformly Lpischitz continuous. Moreover, there is a constant $C>0$
such that
\begin{eqnarray*}
&&|l(t,x,y,v)|+|\phi(x,y)|\leq C(1+x^2+y^2+v^2),\\
&&|\phi_x(x,y)|+|\phi_y(x,y)|\leq C(1+|x|+|y|),\\
&&|l_x(t,x,y,v)|+|l_y(t,x,y,v)|+|l_v(t,x,y,v)|\leq C(1+|x|+|y|+|v|).
\end{eqnarray*}
$f:\R\to\R$ and $g:\R\to\R$ are both continuously differentiable with bounded derivatives $f^\prime(x)$ and $g^\prime(x)$.
\end{description}

All the above mentioned functions in ({\bf H1}), ({\bf H2}) and ({\bf H3}) are deterministic.

Suppose that $u(\cdot)\in {\mathscr{U}}_{ad}$ is an optimal control
process and $x(\cdot)$ is the corresponding state process. We want
to obtain the maximum principle for $u$ and $x$.  Namely, we want to
find necessary conditions that $u$ and $x$ must satisfy.  We shall
follow the same argument as in the previous section. But we can
no longer use Malliavin calculus because of the mean field's
appearance in the state equation \eref{Anew}.  Let $v(\cdot)$ be
another arbitrary    control process in ${\mathscr{U}}_{ad}$.  Since
${\mathscr{U}}_{ad}$ is convex, the following perturbed control
process $u^\varepsilon(\cdot)$ is also an element of
${\mathscr{U}}_{ad}$:
$$u^\varepsilon(t)=u(t)+\varepsilon(v(t)-u(t)),\quad 0\leq \varepsilon\leq 1.$$
We follow all the notations used in the previous section.  For
example, we denote by $x^\varepsilon(\cdot)$ and
$\rho^\varepsilon(\cdot)$ the states of (\ref{Anew}) and
(\ref{rhosolution}) along with the control $u^\varepsilon(\cdot).$ When $\vare=0$, denote $x=x(\cdot)$ and
$\rho=\rho(\cdot)$.
Furthermore, suppose that $v(\cdot) \in {\mathscr{U}}_{ad}$ such
that $v^\prime(\cdot)=v(\cdot)-u(\cdot)\in {\mathscr{U}}_{ad},$
 then $v^\prime(\cdot)+u(\cdot)\in {\mathscr{U}}_{ad}.$

The equation for the derivative $\frac{d}{d\vare}\Big|_{\vare=0}
x^\vare(t)$  will be
\begin{eqnarray}\label{x1new}\left\{%
\begin{array}{lll}
    dx_1(t)&=& \Bigl(b_x(t)x_1(t)+b_y(t)\E_0[x_1(t)]+b_v(t)\bigl(v(t)-u(t)\bigr)\Bigr)dt\\
   && +\Bigl(\sigma_x(t)x_1(t)+\sigma_y(t)\E_0[x_1(t)]
    +\sigma_v(t)\bigl(v(t)-u(t)\bigr)\Bigr)dW_1(t)\\
    &&+\displaystyle\int_{\R_0}\Bigl[\gamma_x(t,z)x_1(t)+\gamma_y(t,z)\E_0[x_1(t)]
    +\gamma_v(t,z)\bigl(v(t)-u(t)\bigr)\Bigr]\tilde{N}(dt,dz),\\
x_1(0)&=&0,
\end{array}%
\right.
\end{eqnarray}
and $\frac{d}{d\vare}\Big|_{\vare=0} \rho^\vare(t)$ will satisfy
\begin{eqnarray}\label{rhosolution1new}\left\{%
\begin{array}{lll}
    d\rho_1(t)&=& \Bigl(\rho_1(t) h(t,x(t))+\rho(t)h_x(t)x_1(t)\Bigr)dY(s),\\
\rho_1(0)&=&0\,,
\end{array}%
\right.
\end{eqnarray}
where while the equation is exactly the same as \eref{rhosolution1}
but with $x_1(t)$ being given by \eref{x1new}. Obviously,
$$\rho_1(t)=\rho(t)\Bigl(\int_0^th_x(s)x_1(s)dY(s)-\int_0^th_x(s)h(s,x(s))x_1(s)ds\Bigr),\quad 0\leq t\leq T.$$ In fact,
we have the following



\begin{Lemma}\label{Lemma2new} Let assumptions (H1) and (H2) hold. Then
\begin{eqnarray*}
\begin{array}{lll}
& \lim\limits_{\varepsilon\rightarrow 0}\E_0\left[\sup\limits_{0\leq
t\leq T}
\left|\frac{x^\varepsilon(t)-x(t)}{\varepsilon}-x_1(t)\right| ^2
\right]=0,\quad \lim\limits_{\varepsilon\rightarrow
0}\sup\limits_{0\leq t\leq T}\E\left[\left|\frac{\rho^\varepsilon(
t)-\rho( t)}{\varepsilon}-\rho_1( t)\right|^2\right] =0.
\end{array}
\end{eqnarray*}
\end{Lemma}

\begin{proof} By Lemma 4.3 in \cite{Yang}, we have $\displaystyle
 \lim\limits_{\varepsilon\rightarrow 0}\E_0\left[\sup\limits_{0\leq
t\leq T}
\left|\frac{x^\varepsilon(t)-x(t)}{\varepsilon}-x_1(t)\right| ^2
\right]=0$.   In order to prove the second equality, we apply the
It\^o formula to
$\eta(t):=\frac{\rho^\varepsilon(t)-\rho(t)}{\varepsilon}-\rho_1(t)$
to obtain
\begin{eqnarray*}\left\{%
\begin{array}{lll}
    d\eta(t)&=& \Bigl(\eta(t)h(t,x^\varepsilon(t))+\rho(t)A^\varepsilon(t)\xi(t)+\rho(t)\bigl(A^\varepsilon(t)-h_x(t)\bigr)x_1(t)\\
   &&+ \rho_1(t)\bigl(h(t,x^\varepsilon(t))-h(t,x(t))\bigr)\Bigr)dY(t),\\
\eta(0)&=&0,
\end{array}%
\right.
\end{eqnarray*}
with $A^\varepsilon(t)=\displaystyle\int_0^1 h_x\bigl(x(t)+\theta \varepsilon(x_1(t)+\xi(t))\bigr)d\theta,
\  \xi(t)=\frac{x^\varepsilon(t)-x(t)}{\varepsilon}-x_1(t).$

Then we have
\begin{eqnarray*}\E\eta^2(t)&=& \E\int_0^T\Bigl(\eta(t)h(t,x^\varepsilon(t))+\rho(t)A^\varepsilon(t)\xi(t)+\rho(t)\bigl(A^\varepsilon(t)-h_x(t)\bigr)x_1(t)\\
   &&+ \rho_1(t)\bigl(h(t,x^\varepsilon(t))-h(t,x(t))\bigr)\Bigr)^2dt\\
   &\leq& K_0\E\int_0^T \eta^2(t)dt + o(\varepsilon),
\end{eqnarray*}
where $K_0>0$ is a constant. Now the  Gronwall  inequality yields
the lemma.
\end{proof}

Since  $u(\cdot)$ is an optimal control,  we have
\[
\frac{d}{d\vare}\Big|_{\vare=0} J(u_\varepsilon(\cdot)) \geq 0\,.
\]
Using  Lemma \ref{Lemma2new} and almost the same argument  as for
the equation (\ref{J}), we obtain
\begin{Lemma}Under (H1), (H2) and (H3), if $u(\cdot)$ is an optimal
control and $v(\cdot)$ is any given control process in ${\mathscr{U}}_{ad}$
such that $v(\cdot)-u(\cdot)\in {\mathscr{U}}_{ad}$, then we have
\begin{eqnarray}\label{variation}
 &&\E\Bigg\{ \int_0^T
\Bigl[\rho(t)\Bigl(l_x(t)x_1(t)+f^\prime(x(t))x_1(t)\E_0[l_y(t)]+l_v(t)(v(t)-u(t))\Bigr) \nonumber\\
&&  +  \rho_1(t)\Bigl(f(x(t))\E_0[l_y(t)]+l(u(t))\Bigr)\Bigr]dt+\rho(T)\phi_xx_1(T)\nonumber\\
&&+\rho(T)x_1(T)g^\prime(x(T))\E_0[\phi_y]  +\rho_1(T)\Bigl(g(x(T))\E_0[\phi_y]+\phi\Bigr)\Bigg\}\geq 0.
\end{eqnarray}
\end{Lemma}

Now we shall write the optimality condition \eref{variation} by a
backward mean field stochastic differential equation.
 For any $u(\cdot)\in
\mathscr{U}_{ad}$ and the corresponding state trajectory $x(\cdot)$,
we define the first order adjoint process $(p(\cdot), q(\cdot),
R(\cdot,\cdot))$ as follows:
\begin{eqnarray}\label{smallp}\left\{%
\begin{array}{lll}
    -dp(t)&=& \Bigl(b_x(t)p(t)+\rho(t)\E[b_y(t)p(t)]+\sigma_x(t) q(t)+\rho(t)\E[\sigma_y(t) q(t)]\\
    &&+\rho(t)\Bigl(h_x(t)Q(t)+l_x(t)+f^\prime(x(t))\E_0[l_y(t)]\Bigr)+\displaystyle\int_{\R_0}\gamma_x(t,z)R(t,z)\mu(dz) \\
    &&+\displaystyle\int_{\R_0}\rho(t)\E[\gamma_y(t,z)R(t,z)]\mu(dz)\Bigr)dt-q(t)dW_1(t)- \displaystyle\int_{\R_0} R(t,z)\tilde{N}(dt,dz),\\
p(T)&=&\rho(T)\Bigl(\phi_x+g^\prime(x(T))\E_0[\phi_y]\Bigr),
\end{array}%
\right.
\end{eqnarray}
where $(P(\cdot), Q(\cdot), G(\cdot,\cdot))$ is defined by
\begin{eqnarray}\label{P}\left\{%
\begin{array}{lll}\label{Q}
    -dP(t)&=& \Bigl(l(u(t))+f(x(t))\E_0[l_y(t)]\Bigr)dt-Q(t)dW_2(t)
       -\displaystyle\int_{\R_0} G(t,z)\tilde{N}(dt,dz),\\
P(T)&=& \phi+g(x(T))\E_0[\phi_y],
\end{array}%
\right.
\end{eqnarray}
which is a mean-field backward stochastic differential equation (BSDE
for short) and from  \cite{Yang} this BSDE admits  unique solution
triplet $(p, q, R)$.  Then we define the usual Hamiltonian
associated with the mean-field stochastic control problem as follows
\begin{eqnarray}\label{Hnew}H(t,x,y,v; p,q,R(\cdot),Q,\rho)&=& b(t,x,y,v)p+
\sigma(t,x,y,v)q+\int_{\R_0}R(t,z)\gamma(t,x,y,v,z)\mu(dz)\nonumber\\
&&+h(t,x)Q+l(t,x,y,v)\rho.
\end{eqnarray}
\begin{theorem}\label{t.4.1}
 Under (H1), (H2) and (H3), if $u(\cdot)$ is an optimal control and
$v(\cdot)$ is any given control process in ${\mathscr{U}}_{ad}$ such that $v(\cdot)-u(\cdot)\in {\mathscr{U}}_{ad}$, then it is necessary to satisfy that
\begin{eqnarray}\label{u1}\E\left[H_v(t,x(t),\E_0[f(x(t))], u(t); p(t),q(t),
R(t,\cdot),Q(t),\rho(t))(v(t)-u(t))\left|\right. \mathscr{F}_t^Y\right]\geq
0,\label{e.h}
\end{eqnarray}
\end{theorem}
 where $(p(\cdot), q(\cdot), R(\cdot,\cdot))$ and $Q(\cdot)$ are the solutions of (\ref{smallp}) and (\ref{P}), respectively.

\begin{proof}  Applying It\^o's formula to $\rho_1(\cdot)P(\cdot)$ and
$p(\cdot)x_1(\cdot),$ we obtain
\begin{eqnarray*}\E\Bigl[\rho_1(T)\Bigl(\phi+g(x(T))\E_0[\phi_y]\Bigr)\Bigr]&=&\E\int_0^T\Bigl[\rho(t)Q(t)h_x(t)x_1(t)
-\rho_1(t)\Bigl(l(u(t))\\
&&+f(x(t))\E_0[l_y(t)]\Bigr)\Bigr]dt
\end{eqnarray*}
and
\begin{eqnarray*}&&\E\Bigl[x_1(T)\rho(T)\Bigl(\phi_x+g^\prime(x(T))\E_0[\phi_y]\Bigr)\Bigr]=\E \int_0^T\Bigl[\Bigl(b_v(t)p(t)+\sigma_v(t)q(t)\\
&&\quad+\int_{\R_0}\gamma_v(t,z)R(t,z)\mu(dz)\Bigr)v(t)-\Bigl(l_x(t)+h_x(t)Q(t)+f^\prime(x(t))\E_0[l_y(t)]\Bigr)\rho(t)x_1(t)\Bigr]dt.
\end{eqnarray*}
Inserting the above two equations into the variational inequality (\ref{variation}), we have
$$\E\int_0^T\Bigl(b_v(t)p(t)+\sigma_v(t)q(t)+\int_{\R_0}\gamma_v(t,z)R(t,z)\mu(dz)+l_v(t)\Bigr)(v(t)-u(t))dt\geq 0,$$
Thus, the proof is completed.
\end{proof}
\begin{remark}\label{r.4.2}
If $u(\cdot)$ is a local minimum for the performance functional $J$
(given by (\ref{performanceP})), in the sense that for all bounded
$v(\cdot)\in \mathscr{U}_{ad},$ there exists an $\delta >0$ such
that $u(\cdot)+\varepsilon v(\cdot)\in \mathscr{U}_{ad}$ for any
$\varepsilon\in (-\delta,\delta)$ and
$${\cal J}(\varepsilon)=J(u(\cdot)+\varepsilon v(\cdot)),\quad \varepsilon\in (-\delta,\delta),$$
attains its minimum at $\varepsilon=0,$   then  (\ref{variation})
and hence \eref{e.h}  are    identities.
\end{remark}

\vspace{0.2cm} \noindent{\bf  Applications} \vspace{0.2cm}

We aim to illustrate Theorem \ref{t.4.1}  by a linear-quadratic (LQ)
example as in Section  3. Consider the flowing LQ optimal control
problem with partial information. Namely, we want to minimize $
J(v(\cdot)) $, where
\begin{eqnarray*}\label{examplenew}
&&J(v(\cdot))=\frac{1}{2} \E_0\left\{\int_0^T \Bigl[L(t)\Bigl( x^v(t)-\E_0[x^v(t)]\Bigr)^2+O(t)\Bigl(v(t)-M(t)\Bigr)^2\Bigr]dt\right.\nonumber\\
&&\quad\quad\quad\quad\left.+N\Bigl(x^v(T)-\E_0[x^v(T)]\Bigr)^2\right\}
\end{eqnarray*}
subject to
\begin{eqnarray*}\label{xnewnew}\left\{%
\begin{array}{lll}
    dx^v(t)&=& \Bigl(A(t)x^v(t)+B(t)\E_0[x^v(t)]+C(t)v(t)\Bigr)dt+\Bigl(D(t)x^v(t)+E(t)\E_0[x^v(t)]\\
   && +F(t)v(t)\Bigr)dW_1(t)
 +\displaystyle\int_{\R_0}\Bigl(S(t,z)x^v(t-)+K(t,z)\E_0[x^v(t-)]\\
 &&
 +I(t,z)v(t)\Bigr)\tilde{N}(dt,dz),\\
x^v(0)&=&x_0\in \R. \\
\end{array}%
\right.
\end{eqnarray*}
The observation is $ dY(t)= h(t,x^v(t))dt+d{W}_2(t),\ \  Y(0)=0.$

\vspace{0.2cm} Here $L(\cdot)\geq 0$, $O(\cdot)>0$,
$\frac{1}{O(\cdot)}$, $M(\cdot)$, $A(\cdot)$, $B(\cdot)$,
$C(\cdot)$, $D(\cdot)$, $E(\cdot)$, $F(\cdot)$, $S(\cdot,\cdot)$,
$K(\cdot,\cdot)$, $I(\cdot,\cdot)$ are uniformly bounded and
deterministic; $N\geq 0$ is a constant. $h$ satisfies the assumption
(H2). Theorem  \ref{t.4.1} is valid. Thus, we define the Hamiltonian
as below.
\begin{eqnarray}\label{Hnewnew}&&H(t,x,y,v; p,q,R(\cdot),Q,\rho)= \bigl(A(t)x+B(t)y
+C(t)v\bigr)p+ \bigl(D(t)x+E(t)y+F(t)v\bigr)q\nonumber\\
&&\quad\quad\quad\quad+\int_{\R_0}R(t,z)\Bigl(S(t,z)x+K(t,z)y+I(t,z)v\Bigr)\mu(dz)+h(t,x)Q\nonumber\\
&&\quad\quad\quad\quad+\frac{1}{2}\rho L(t)(x-y)^2+\frac{1}{2}\rho O(t)(v-M(t))^2.
\end{eqnarray}
The corresponding adjoint processe $(p(\cdot), q(\cdot), R(\cdot,\cdot))$  is defined as follows:
\begin{eqnarray}\label{smallpnew}\left\{%
\begin{array}{lll}
    -dp(t)&=& \Bigl[A(t)p(t)+\rho(t)\E[B(t)p(t)]+D(t) q(t)+\rho(t)\E[E(t) q(t)]\\
      &&+\rho(t)\Bigl(h_x(t)Q(t)+L(t)\bigl(x(t)-\E_0 [x(t)]\bigr)\Bigr)\\
     &&+\displaystyle\int_{\R_0}S(t,z)R(t,z)\mu(dz)+\displaystyle\int_{\R_0}\rho(t)\E[K(t,z)R(t,z)]\mu(dz)\Bigr]dt\\
  &&-q(t)dW_1(t)- \displaystyle\int_{\R_0} R(t,z)\tilde{N}(dt,dz),\\
p(T)&=&N\rho(T)\Bigl(x(T)-\E_0[x(T)]\Bigr),
\end{array}%
\right.
\end{eqnarray}
where $(P(\cdot), Q(\cdot), G(\cdot,\cdot))$ is defined by
\begin{eqnarray*}\label{Pnew}\left\{%
\begin{array}{lll}
    -dP(t)&=& \frac{1}{2}\Bigl[L(t)\Bigl( x(t)-\E_0[x(t)]\Bigr)^2+O(t)\Bigl(u(t)-M(t)\Bigr)^2\Bigr]dt\\
    &&-Q(t)dW_2(t)
       -\displaystyle\int_{\R_0} G(t,z)\tilde{N}(dt,dz),\\
P(T)&=& \frac{1}{2}N\Bigl(x(T)-\E_0[x(T)]\Bigr)^2.
\end{array}%
\right.
\end{eqnarray*}
By Remark \ref{r.4.2}  if $u(\cdot)$ is local minimum, then it is
necessary to satisfy
\begin{eqnarray*}\E\left[\rho(t)O(t)(u(t)-M(t))+C(t)p(t)+F(t)q(t)+\int_{\R_0}R(t,z)I(t,z)
\mu(dz)\left|\right. \mathscr{F}_t^Y\right]= 0,\end{eqnarray*}
where $(p(\cdot), q(\cdot), R(\cdot,\cdot))$ is the solution to (\ref{smallpnew}). Then
\begin{eqnarray}\label{u}u(t)=-\frac{1}{\rho(t)O(t)}\E\left[C(t)p(t)+F(t)q(t)
+\int_{\R_0}R(t,z)I(t,z)\mu(dz)\left|\right. \mathscr{F}_t^Y\right]+M(t).
\end{eqnarray}
\begin{proposition} If $u(\cdot)$ is an optimal control strategy, then it is necessary to satisfy (\ref{u}).
\end{proposition}



\end{document}